\documentclass[12pt]{amsart}
\usepackage{amssymb,amsfonts,amsmath,amsthm,cite,verbatim}
\usepackage{xcolor}
\usepackage{tikz}
\usepackage{graphicx}
\usepackage[left=2.5cm, right=2.5cm, top=3.5cm, bottom=3.5cm]{geometry}
\usepackage{booktabs}
\usepackage{array}
\usepackage{hyperref}
\usepackage[normalem]{ulem}
\usepackage{appendix}
\usepackage{indentfirst}
\allowdisplaybreaks[4]

\newtheorem{theorem}{Theorem}[section]
\newtheorem{theo}[theorem]{Theorem}
\newtheorem{prop}[theorem]{Proposition}
\newtheorem{lem}[theorem]{Lemma}
\newtheorem{rem}[theorem]{Remark}

\newtheorem{cor}[theorem]{Corollary}
\newtheorem{ex}[theorem]{Example}

\newtheorem{conj}[theorem]{Conjecture}

\makeatletter \@addtoreset{equation}{section}

\DeclareMathOperator*{\CT}{CT}
\DeclareMathOperator*{\res}{Res}

\newcommand{\NN}{\mathbb{N}}
\newcommand{\RR}{\mathbb{R}}
\def\x{\boldsymbol{x}}
\def\y{\boldsymbol{y}}
\def\B{\mathcal{B}}
\def\H{\mathcal{H}}
\def\m{\boldsymbol{m}}
\def\b{\boldsymbol{b}}

\author{Guoce Xin$^{1}$ and Chen Zhang$^{2,*}$}

\address{ $^{1, 2}$School of Mathematical Sciences,  Capital Normal University,
 Beijing 100048,  PR China}

\email{$^1$\texttt{guoce\_xin@163.com}\\ $^2$\texttt{ch\_enz@163.com}}

\thanks{$*$ Corresponding author.}

\begin{document}
\title{A variation of the Morris constant term}

\date{October 23, 2024}
%\date{\today}

\begin{abstract}
The Morris constant term identity is important due to its equivalence with the well-known Selberg integral. We find a variation of the Morris constant term, denoted $h_n(t)$, in the study of the Ehrhart polynomial $H_n(t)$ of the $n$-th Birkhoff polytope, which consists of all doubly stochastic matrices of order $n$. The constant term $h_n(t)$ corresponds to a particular constant term in the study of $H_n(t)$. We give a characterization of $h_n(t)$ as a polynomial of degree $(n-1)^2$ with additional nice properties involving the Morris constant term. We also construct a recursion for $h_n(t)$ using a similar technique for the proof of the Morris constant term identity by Baldoni-Silva and Vergne, and by Xin. Using this method, we have produced explicit formulas for $h_n(t)$ for $3 \le n \le 29$ without difficulty.
\end{abstract}

\maketitle

\noindent
\begin{small}
\end{small}

\noindent
\begin{small}
\end{small}

\section{Introduction}
We are interested in the following form of Morris constant term identity, which was first proved by Baldoni-Silva and Vergne using total residue
in \cite{2001-Baldoni-Silva}, and then given a simplified proof in \cite{2004-Xin-phd}.
\begin{theo}[\cite{2001-Baldoni-Silva,2004-Xin-phd}]\label{theo-Morris}
If $k_1, k_2, k_3 \in \NN$ and $k_1 + k_2 \ge 2$, then
\[
\CT_{\x} M(n; k_1, k_2, k_3) = \prod_{j=0}^{n-1} \frac{\Gamma(1 + \frac{k_3}{2}) \Gamma(k_1 + k_2 - 1 + (n + j -1) \frac{k_3}{2})}{\Gamma(1 + (j + 1) \frac{k_3}{2}) \Gamma(k_1 + j \frac{k_3}{2}) \Gamma (k_2 + j \frac{k_3}{2})}.
\]
where $\displaystyle\CT_{\x}f$ denotes the constant term of the Laurent series $f$ with respect to the variables $\x:=(x_1, x_2, \dots,x_n)$, and
\begin{equation}\label{e-Morris}
M(n; k_1, k_2, k_3) := \frac{1}{\prod_{i=1}^n x_i^{k_1 - 1} \prod_{i=1}^n (1-x_i)^{k_2} \prod_{1 \le i<j \le n} (x_i-x_j)^{k_3}}.
\end{equation}
In particular, we agree that $M(0; k_1, k_2, k_3) = 1$.
\end{theo}
Here and what follows, rational functions in $\x$ are explained as their Laurent series expansions according to $1> x_1> x_2 >\cdots>x_n>0$, unless specified otherwise.

Morris's original constant term identity (in his PhD thesis \cite{1982-Morris}) refers to the negative even $k_3$ case. It is equivalent to the well-known Selberg integral \cite{1944-Selberg}. See \cite{2008-Forrester-Selberg} for a nice overview of the importance of the Selberg integral.
Zeilberger \cite{1999-Zeilberger} first considered the case $(k_1,k_2,k_3)=(1,2,1)$ of Theorem \ref{theo-Morris}, and proved that the volume of the Chan-Robbins-Yuen polytope equals the product of Catalan numbers. Theorem \ref{theo-Morris} was further studied in \cite{2021-symMorris}.

In this paper, we mainly study a variation of the Morris constant term defined in \eqref{e-CTH21}.
This constant term arises as a particular case of the Ehrhart polynomial $H_n(t)$ of the $n$-th Birkhoff polytope $\B_n$.

For a positive integer $n$, $\B_n$ is usually described as the set of all $n \times n$ \emph{doubly stochastic} matrices, where a matrix $B \in \RR^{n \times n}$ is said to be doubly stochastic if all elements of $B$ are nonnegative and the sum of each row and each column is equal to $1$.
The $H_n(t)$ counts the number of the nonnegative integral matrices such that every row sum and every column sum equal to the nonnegative number $t$. This has been studied by many researchers, e.g. \cite{2003-BeckPixton,Clara,John,Per,Diaconis,Loera,2001-Baldoni-Silva,2022-CTTypeA}.
It was conjectured by Anand, Dumir and Gupta \cite{1966-Anand-Hn} and proved by Ehrhart \cite{1973-Ehrhart} and Stanley \cite{1973-Stanley-Hn} that $H_n(t)$ is a polynomial in $t$ of degree $(n-1)^2$ for any fixed positive integer $n$. And the following theorem holds.
\begin{theo}[\cite{1973-Stanley-Hn,1976-Stanley-Hn}]\label{theo-Stanley}
For any positive integer $n$, we have
\[
\sum_{t \ge 0} H_n(t) y^t = \frac{1}{(1 - y)^{(n-1)^2+1}} \sum_{i=0}^{(n-1)(n-2)} a_i y^i,
\]
where the $a_i$'s are nonnegative integers satisfying $a_0 = 1$ and $a_i = a_{(n-1)(n-2)-i}$ for all $i$.
\end{theo}
It is the first conjecture stated in \cite{1983-Stanley} that the $(a_0, a_1, \dots, a_{(n-1)(n-2)})$ is unimodal, i.e., $a_0 \le a_1 \le \cdots \le a_{\lfloor (n-1)(n-2)/2 \rfloor}$. Athanasiadis \cite{2005-Athanasiadis} proved this conjecture using McMullen's $g$-theorem \cite{1971-gtheorem}.

The formulas $H_1(t) = 1$ and $H_2(t) = t + 1$ are trivial. The first nontrivial case was computed by MacMahon \cite{1960-MacMahon} as $H_3(t) = 3 \binom{t+3}{4} + \binom{t+2}{2}$.
The record is kept by Beck and Pixton \cite{2003-BeckPixton}, who obtained an explicit expression for $H_9(t)$ and the leading coefficient of $H_{10}(t)$ (corresponding to the volume of $\B_{10}$) using residue computation.
For a positive integer $n$, $H_n(t)$ can be written as \cite{2003-BeckPixton,2022-CTTypeA}
\[
H_n(t)= \sum_{\m} \binom{n}{\m} \CT_{\x} \H^{\m},
\]
where $\m := (m_1, m_2, \dots, m_n)$ ranges over all weak compositions of $n$ with $n$ components (that is, the $m_i$'s are nonnegative integers such that $\sum_{i=1}^nm_i=n$), and
\[
\H^{\m} := \prod_{i=1}^{n}\frac{ x_i^{(m_i-1)t}}{\prod_{j=1,j\neq i}^n (1-x_j/x_i)^{m_i}}.
\]
Among all these $\m$, the case $\m = (2, 1^{n-2}, 0)$ is the hardest to evaluate using usual methods.
We observe that $\CT_{\x} \H^{(2, 1^{n-2}, 0)}$ is similar to the Morris constant term. To be precise,
\[
\H^{(2, 1^{n-2}, 0)} =(-1)^{\frac{(n-1)(n-2)}{2}} \frac{(x_n/x_1)^{-t} \prod_{i=2}^{n - 1} (x_i/x_1)^{2 i - n} \prod_{i=1}^{n-1} (1 - x_n/x_i)}{\prod_{i=2}^{n} (1 - x_i/x_1)^3 \prod_{2\le i < j \le n} (1-x_j/x_i)^2}.
\]
Since $\H^{(2, 1^{n-2}, 0)}$ is homogeneous, we may set $x_1 = 1$ when computing $\CT_{\x} \H^{(2, 1^{n-2}, 0)}$. That is, we can write
\begin{equation}\label{e-CTH21}
h_n(t) := (-1)^{\frac{(n-1)(n-2)}{2}} \CT_{\x} \H^{(2,1^{n-2},0)} = \CT_{\x} \bar \H_n,
\end{equation}
where
\begin{equation}\label{e-barH21}
\bar \H_n := \frac{(1-x_n) \prod_{i=2}^{n-1} x_i^{n-1} \prod_{i=2}^{n-1} (x_i-x_n)}{x_n^t \prod_{i=2}^n (1-x_i)^{3} \prod_{2 \le i<j \le n} (x_i-x_j)^{2}}.
\end{equation}

Now we present our main result as follows. It will be proved in Section \ref{s-proof}.
\begin{theo}\label{theo-main}
Let $h_n(t)$ and $M(n; k_1, k_2, k_3)$ be as in \eqref{e-CTH21} and \eqref{e-Morris}, respectively.
For any $n \ge 2$, we have
\begin{enumerate}
  \item $h_n(t)$ is a polynomial in $t$ of degree $(n-1)^2$ with leading coefficient $ \frac{1}{(n-1)^2!} \CT_{\x} M(n-1; 1, 1, 2)$.

  \item $h_n(t) = (-1)^{n-1} h_n(-t-n)$.

  \item Let $N=\big\lfloor \frac{(n-1)^2}{4} \big\rfloor$. Then $h_n(t) = P_n(t)\prod_{i=-N - n + 1 }^{N - 1} (t-i)$, where $P_n(t)$ is a polynomial in $t$ of degree $(n-1)(n-2) - 2 N$.

  \item Let $N$ be as in part (3). Then
  \[
  h_n(N) = \Big(\CT_{\x} M\Big( \Big\lfloor \frac{n-1}{2} \Big\rfloor; 2 + \chi(n \text{ is even}), 1, 2 \Big) \Big)^2,
  \]
  where $\chi(\text{true}) = 1$ and $\chi(\text{false}) = 0$.
\end{enumerate}
\end{theo}

It is trivial that $h_2(t) = t+1$. In what follows, we always assume $n \ge 3$.

According to Theorem \ref{theo-main}, for a fixed $n$, we can construct the polynomial $h_n(t)$ by the method of undetermined coefficients. By part (1)
we need the values of $h_n(t)$ at $(n-1)^2$ distinct $t$'s; Part (2) cuts this number by half; Part (3) gives some roots of $h_n(t)$; Part (4) gives $h_n(N)$.
This leaves us to compute $h_n(t)$ at $\binom{n-1}{2} - N-1$ different  $t>N$. Direct computation of these $h_n(t)$ is already hard when $n\ge 8$.
Instead, we will derive a recursion for $h_n(t)$ in Section \ref{s-recursion}, at least for $n\leq 29$. The recursion is derived using a relation established using a similar technique for the proof of the Morris constant term identity.

In particular, we obtain $h_{10} (t) = P_{10}(t) \prod_{i = -29}^{19} (t-i) $ with
\[
P_{10}(t)= \frac{321375112314406569619200000}{81!}\ t^{32}+ \cdots,
\]
which is too long to put here.
The explicit formulas for $h_n(t)$ for $3 \le n \le 29$ are available at \cite{datalink}. The data shows that $P_n(t)$ is divisible by $t(t+n)$ when $n\geq 4$.

The rest of this paper is structured as follows.
In Section \ref{s-tool}, we introduce two versions of Jacobi's change of variable formula, and a constant term concept in \cite{2022-CTTypeA}. These will be the main tools
in our proof of Theorem \ref{theo-main}.
In Section \ref{s-proof}, we complete the proof of Theorem \ref{theo-main} by dividing it into some lemmas.
In Section \ref{s-recursion}, we introduce a new set of constant terms $D_n(\ell, t, k_1, k_2, k_3)$ involving elementary symmetric functions. Then we obtain a relation among them and use these relations to derive a recursion for $h_n(t)$. By combining with Theorem \ref{theo-main}, we can compute the explicit formula for $h_n(t)$ for any $n$. We provide Examples \ref{ex-neq4} and \ref{ex-h5} to explain our method. We also give a conjectured nice formula for a tri-diagonal determinant (See Conjecture \ref{conj-detC}).
Section \ref{s-conj} gives a conjecture about the generating function of $h_n(t)$.

\section{Jacobi's change of variable formula and a constant term concept}\label{s-tool}
In this section, we introduce Jacobi's change of variable formulas and a basic constant term concept.
These will be used in our proof.

We use both constant term and residue, though they are equivalent in some sense.
For a Laurent series $f(x)$ at $x=x_0$, we denote the coefficient of $(x - x_0)^j$ by $[(x - x_0)^j] f(x)$, and
the residue of $f(x)$ at $x=x_0$ by
\[
\res_{x =x_0} f(x) := [(x - x_0)^{-1}] f(x).
\]
In particular, we simply write $\res_{x} f(x) := \res_{x=0} f(x) = [x^{-1}] f(x)$ if there is no ambiguity. We have the relation
\[
\CT_x x f(x) = \res_x f(x).
\]

\subsection{Jacobi's change of variable formula}
We need the following Jacobi's change of variable formula in the one variable case.
\begin{lem}[\cite{1830-Jacobi}]\label{lem-res1830}
Let $c$ be a complex number. Suppose $g(y)$ is holomorphic in a neighborhood of $y=c$ and suppose $f(x)$ is meromorphic in a neighborhood of $x=g(c)$. If $\big( \partial_y g(y) \big) \big|_{y=c} \neq 0$, then
\[
\res_{x=g(c)} f(x) = \res_{y=c} f(g(y)) \partial_y g(y),
\]
where $\partial_y := \frac{\partial}{\partial y}$.
\end{lem}

\begin{lem}\label{lem-residue-1}
Let $F(x_1, x_2, \dots, x_n)$ be a rational function. Then we have
\[
  \CT_{\x} F(x_1, x_2, \dots, x_n) = \CT_{\y} \frac{1}{\prod_{i=1}^{n} (1-y_i)} F \Big(-\frac{y_1}{1-y_1}, -\frac{y_2}{1-y_2}, \dots, -\frac{y_n}{1-y_n} \Big).
\]
\end{lem}
\begin{proof}
The lemma follows by applying Lemma \ref{lem-res1830} with $x_i = - \frac{y_i}{1-y_i}$ successively for all $i$:
\begin{align*}
\CT_{\x} F(x_1, x_2, \dots, x_n) &= \res_{\x} \frac{1}{\prod_{i=1}^n x_i} F(x_1, x_2, \dots, x_n) \\
&= \res_{\y} \frac{\prod_{i=1}^{n} \partial_{y_i} (- \frac{y_i}{1-y_i})}{\prod_{i=1}^n (- \frac{y_i}{1-y_i})} F \Big(-\frac{y_1}{1-y_1}, -\frac{y_2}{1-y_2}, \dots, -\frac{y_n}{1-y_n} \Big) \\
&= \res_{\y}  \frac{1}{\prod_{i=1}^{n} y_i (1-y_i)} F \Big(-\frac{y_1}{1-y_1}, -\frac{y_2}{1-y_2}, \dots, -\frac{y_n}{1-y_n} \Big) \\
&= \CT_{\y} \frac{1}{\prod_{i=1}^{n} (1-y_i)} F \Big(-\frac{y_1}{1-y_1}, -\frac{y_2}{1-y_2}, \dots, -\frac{y_n}{1-y_n} \Big). \qedhere
\end{align*}
\end{proof}

Let $F(\x)$ be a formal Laurent series.
We define the action of a nonsingular integer matrix \( W = (w_{ij})_{n \times n} \) by
$$WF(\x) := W \circ F(\x) = F(y_1, y_2, \dots, y_n), $$
where $y_i = \x^{W e_i} = x_1^{w_{1i}} \cdots x_n^{w_{ni}} $, and $e_i$ is the standard $i$-th unit vector for each $i$.
Notably, \( W\x^{\alpha} = \x^{W\alpha} \).

For the multi-variable case, we need the next version of Jacobi's change of variable formula.
\begin{lem} \label{lem-changevars}
Suppose \( W \) is a nonsingular integer matrix and \( F(\x) \) is a formal Laurent series in \( \x \). Then we have
$$\CT_{\x} F(\x) = \CT_{\x} (W F(\x)),
$$
provided both sides converge.
\end{lem}
\begin{proof}
By linearity, we can assume $ F(\x) = \x^\alpha $. Then $W F(\x) = \x^{W\alpha}$. Since $W \alpha = \mathbf{0}$ if and only if $\alpha = \mathbf{0}$, the lemma follows.
\end{proof}
See \cite{2004-Xin-phd} for a general version of the above lemma, where the convergent problem is addressed there.
In our applications, the convergence will be easy to check.

\subsection{A constant term concept}
For positive integers $q_1, q_2, \dots, q_n$, let
\begin{equation}\label{e-fx}
  f(x)=\frac{L(x)}{\prod_{i=1}^n (1 - x/u_i)^{q_i}},
\end{equation}
where $L(x)$ is a Laurent polynomial in $x$, and the $u_i$'s are free of $x$.
Assume the PFD (short for partial fraction decomposition) of $f(x)$ with respect to $x$ is given by
\begin{equation}\label{e-fx-PFD}
f(x)=L_1(x) + \sum_{i=1}^n \frac{A_i(x)}{(1 - x/u_i)^{q_i}},
\end{equation}
where $L_1(x)$ is a Laurent polynomial in $x$, and $A_i(x)$ is a polynomial in $x$ of degree less than $q_i$ for each $i$.
Define
\begin{equation}\label{e-CTxui}
\CT_{x = u_i} f(x) := A_i(0).
\end{equation}
Then we have the following lemmas.
\begin{lem}\label{lem-contri}
Let $f(x)$ be as in \eqref{e-fx} with PFD as in \eqref{e-fx-PFD}. Then
\begin{equation}\label{e-CTfx}
\CT_x f(x) = \CT_x L_1(x) + \sum_{i=1}^{n} \chi(x/u_i < 1) \CT_{x = u_i} f(x).
\end{equation}
In particular, if $f(x)$ is proper in $x$, i.e., the degree in the numerator less than the degree in the denominator, then $L_1(x)$ only contains negative powers in $x$. Consequently,
\begin{equation}\label{e-CTfx-proper}
\CT_x f(x) = \sum_{i=1}^{n} \chi(x/u_i < 1) \CT_{x = u_i} f(x).
\end{equation}
\end{lem}
\begin{proof}
Equation \eqref{e-CTfx} follows by the following observation:
\[
\CT_x \frac{A_i(x)}{(1 - x/u_i)^{q_i}} = \begin{cases}
                                             A_i(0), & \mbox{ if } x/u_i <1; \\
                                             0, & \mbox{ if } x/u_i >1.
                                           \end{cases}
\]
This is because when $x/u_i > 1$, the series expansion
\[
\frac{A_i(x)}{(1 - x/u_i)^{q_i}} = \frac{A_i(x)}{(-x/u_i)^{q_i} (1 - u_i/x)^{q_i}} = (- u_i/x)^{q_i} A_i(x) \sum_{l \ge 0} \binom{l+q_i-1}{l} (u_i/x)^l
\]
only contains negative powers of $x$.

Equation \eqref{e-CTfx-proper} holds by $\CT_x L_1(x) = 0$ when $f(x)$ is proper in $x$.
\end{proof}

\begin{lem}\label{lem-CT2res}
Let $f(x)$ be as in \eqref{e-fx} with PFD as in \eqref{e-fx-PFD}. Then
\[
\CT_{x = u_i} f(x) = - \res_{x = u_i} x^{-1} f(x).
\]
\end{lem}
\begin{proof}
Suppose $A_i(x) = \sum_{l=0}^{q_i-1} a_{i,l} x^l$, where the $a_{i,l}$'s are not all zero.
By the definition of $\CT\limits_{x = u_i} f(x)$ (See \eqref{e-CTxui}), it is sufficient to show that
\[
\res_{x = u_i} x^{-1} f(x) = [(x-u_i)^{-1}] x^{-1} f(x) = - a_{i,0}.
\]
For any integer $r$, we can expand $x^r$ at $x=u_i$ as follows:
\[
x^r = u_i^r (1 - (1 - x/u_i))^r = \begin{cases}
                                    \sum_{k=0}^{r} \binom{r}{k} u_i^{r-k} (x - u_i)^k, & \mbox{if } r \ge 0; \\
                                    \sum_{k \ge 0} \binom{k-r-1}{k} (-1)^k u_i^{r-k} (x - u_i)^k, & \mbox{if } r < 0.
                                  \end{cases}
\]
It is clear that
\[
\res_{x = u_i} x^{r} = 0,
\]
and
\[
\res_{x = u_i} \frac{x^r}{(1 - x/u_i)^{q_i}} = (- u_i)^{q_i} [(x-u_i)^{q_i - 1}]\; x^r
= \begin{cases}
 -1, & \mbox{if } r = -1; \\
 0, & \mbox{if } 0 \le r \le q_i - 2.
\end{cases}
\]
By linearity, we have
\begin{align*}
\res_{x = u_i} x^{-1} f(x) &= \res_{x = u_i} \Big( x^{-1} L_1(x) + \sum_{j=1, j \neq i}^n \frac{x^{-1} A_j(x)}{(1 - x/u_j)^{q_j}} \Big) + \res_{x = u_i} \frac{x^{-1} A_i(x)}{(1 - x/u_i)^{q_i}} \\
&= 0 + \sum_{l=0}^{q_i-1} a_{i,l} \res_{x = u_i} \frac{x^{l-1}}{(1 - x/u_i)^{q_i}} = -a_{i,0},
\end{align*}
as desired.
\end{proof}

\begin{lem}\label{lem-residue-2}
Let $F(x_1, x_2, \dots, x_n)$ be a rational function. We have
\[
\CT_{x_n = 1} \cdots \CT_{x_2 = 1} \CT_{x_1 = 1}  F(x_1, x_2, \dots, x_n) = \CT_{\y} \frac{\prod_{i=1}^{n} y_i}{\prod_{i=1}^{n} (1-y_i)} F(1-y_1, 1- y_2, \dots, 1-y_n)
\]
with order $1 > y_n > y_{n-1} > \cdots > y_1 > 0 $ when expanding as a Laurent series.
\end{lem}
\begin{proof}
We successively make the change of variables $x_i = 1-y_i$ for $i=1,2,\dots, n$ to obtain
\begin{align*}
\CT_{x_n = 1} \cdots \CT_{x_2 = 1} \CT_{x_1 = 1} F &= \res_{x_n = 1} \cdots \res_{x_2 = 1} \res_{x_1 = 1} \frac{F(x_1, x_2, \dots, x_n)}{\prod_{i=1}^{n} (- x_i)} & &(\text{by Lemma \ref{lem-CT2res}}) \\
&= \res_{y_n} \cdots \res_{y_2} \res_{y_1} \frac{F(1-y_1, 1- y_2, \dots, 1-y_n)}{\prod_{i=1}^{n} (1-y_i)} & &(\text{by Lemma \ref{lem-res1830}}) \\
&= \CT_{\y} \frac{\prod_{i=1}^{n} y_i}{\prod_{i=1}^{n} (1-y_i)} F(1-y_1, 1- y_2, \dots, 1-y_n).
\end{align*}
The order is naturally $1 > y_n > y_{n-1} > \cdots > y_1 > 0 $.
\end{proof}

The authors and their collaborators \cite{2022-CTTypeA} provided an explicit formula for $A_i(0)$. We restate it here. See \cite[Lemma 2.1]{2022-CTTypeA} for the proof. For convenience, we denote
\begin{equation}\label{e-defin-g}
g_i (x) = f(x)(1 - x/u_i)^{q_i} = \frac{L(x)}{\prod_{l=1,l\neq i}^n (1 - x/u_l)^{q_l}}
\end{equation}
for each $i \in \{1,2,\dots,n\}$.

\begin{lem}[\cite{2022-CTTypeA}]\label{lem-Ai0}
Let $f(x)$, $A_i(x)$ and $g_i(x)$ be as in \eqref{e-fx}, \eqref{e-fx-PFD} and \eqref{e-defin-g}, respectively.
Then
\[
\CT_{x = u_i}f(x) = A_i(0)= \frac{(-1)^{q_i-1}}{(q_i-1)!} \Big(\partial_w^{q_i-1} \frac{g_i(w u_i)}{w} \Big) \Big|_{w=1}.
\]
In particular, if $q_i=1$, then
\begin{equation}\label{e-u-A10-1}
A_i(0) = g_i(u_i).
\end{equation}
\end{lem}

\begin{lem}\label{lem-CT2LRe}
For integers $p_1, p_2, \dots, p_n$, positive integers $q_{0,1}, q_{0,2}, \dots, q_{0,n}$, and nonnegative integers $q_{i,j}$ ($1\le i<j \le n$), let
\[
F = \frac{\prod_{i=1}^{n} x_i^{p_i}}{\prod_{i=1}^n (1-x_i)^{q_{0,i}} \prod_{1\le i<j \le n} (x_i-x_j)^{q_{i,j}} }
\]
be proper in $x_i$ for each $1 \le i \le n$, i.e., $p_i < q_{0,i} + \sum_{j=1}^{i-1} q_{j,i} + \sum_{j=i+1}^{n} q_{i,j}$.
Then
\[
\CT_{\x} F = \CT_{x_n = 1} \cdots \CT_{x_2 = 1} \CT_{x_1 = 1} F.
\]
\end{lem}
\begin{proof}
By Lemmas \ref{lem-contri} and \ref{lem-Ai0}, we have
\begin{align*}
\CT_{x_1} F &=\CT_{x_1 = 1} F \\
&=\frac{(-1)^{q_{0,1}-1}}{(q_{0,1}-1)!} \Big(\partial_w^{q_{0,1}-1} \frac{w^{p_1 - 1}}{\prod_{i=2}^n (w-x_i)^{q_{1,i}}} \Big) \Big|_{w=1} \cdot \frac{\prod_{i=2}^{n} x_i^{p_i}}{\prod_{i=2}^n (1-x_i)^{q_{0,i}} \prod_{2\le i<j \le n} (x_i-x_j)^{q_{i,j}}} \\
&= (-1)^{q_{0,1}-1} \Bigg(\sum_{\substack{\sum_{i=1}^n l_{i} = q_{0,1}-1 \\ l_i \ge 0}} \frac{\partial_w^{l_1}}{l_1!} w^{p_1 - 1} \prod_{i=2}^{n} \frac{\partial_w^{l_i}}{l_i!} (w - x_i)^{-q_{1,i}} \Bigg) \Bigg|_{w=1} \\
& \quad\qquad \cdot \frac{\prod_{i=2}^{n} x_i^{p_i}}{\prod_{i=2}^n (1-x_i)^{q_{0,i}} \prod_{2\le i<j \le n} (x_i-x_j)^{q_{i,j}}} \\
&= (-1)^{q_{0,1}-1} \sum_{\substack{\sum_{i=1}^n l_{i} = q_{0,1}-1 \\ l_i \ge 0}} \binom{p_1 - 1}{l_1} \prod_{i=2}^{n} \binom{-q_{1,i}}{l_i} \cdot G_{l_1, l_2, \dots, l_n},
\end{align*}
where each
\[
G_{l_1, l_2, \dots, l_n} = \frac{\prod_{i=2}^{n} x_i^{p_i}}{\prod_{i=2}^n (1-x_i)^{q_{0,i} + q_{1,i} + l_i} \prod_{2\le i<j \le n} (x_i-x_j)^{q_{i,j}}}
\]
is proper in $x_i$ for any $2 \le i \le n$.
The lemma then follows by repeating the above process.
\end{proof}

Applying Lemma \ref{lem-CT2LRe} to $\bar \H_n$ (as defined in \eqref{e-barH21}) yields the following corollary.

\begin{cor}\label{cor-barH2xi1}
Let $h_n(t)$ and $\bar \H_n$ be as in \eqref{e-CTH21} and \eqref{e-barH21}, respectively. Then
\[
h_n(t) = \CT_{x_n = 1} \cdots \CT_{x_3 = 1} \CT_{x_2 = 1} \bar \H_n.
\]
\end{cor}

We also need the following lemma, which is a consequence of Theorem \ref{theo-Morris}.
\begin{lem}\label{lem-Morris}
Let $M(n; k_1, k_2, k_3)$ be as in \eqref{e-Morris}. If $k_2 = 1$, then
\[
\CT_{\x} M(n + 1; k_1, 1, k_3) = \CT_{\x} M(n; k_1, k_3 + 1, k_3).
\]
\end{lem}
\begin{proof}
We give a constant term proof. We first compute
\begin{align*}
\CT_{x_1} M(n + 1; k_1, 1, k_3) &= \CT_{x_1 = 1} M(n + 1; k_1, 1, k_3) \hspace{3.8cm} (\text{by \eqref{e-CTfx-proper}}) \\
&= M(n + 1; k_1, 1, k_3) \cdot (1-x_1) \big|_{x_1 = 1} \hspace{2cm} (\text{by \eqref{e-u-A10-1}}) \\
&= \frac{1}{\prod_{i=2}^{n+1} x_i^{k_1 - 1} \prod_{i=2}^{n+1} (1-x_i)^{k_3 + 1} \prod_{2 \le i<j \le n+1} (x_i-x_j)^{k_3}}.
\end{align*}
Relabelling the variables by $x_i = x_{i-1}$ for all $2 \le i \le n+1$ yields
\[
\CT_{x_1} M(n + 1; k_1, 1, k_3) = M(n; k_1, k_3 + 1, k_3).
\]
Then the lemma follows.
\end{proof}

Theorem \ref{theo-Morris} presents the following nontrivial symmetry.
\begin{cor}\label{cor-symMorris}
Let $M(n; k_1, k_2, k_3)$ be as in \eqref{e-Morris}. Then
\[
\CT_{\x} M(n; k_1, k_2, k_3) = \CT_{\x} M(n; k_2, k_1, k_3).
\]
\end{cor}
\begin{proof}
Clearly $M(n; k_1, k_2, k_3)$ is proper in $x_i$ for each $i$. Then
\begin{align*}
\CT_{\x} M(n; k_1, k_2, k_3) &= \CT_{x_n = 1} \cdots \CT_{x_2 = 1} \CT_{x_1 = 1} M(n; k_1, k_2, k_3) & & (\text{by Lemma \ref{lem-CT2LRe}}) \\
&= \CT_{\y} \frac{1}{\prod_{i=1}^n (1 - y_i)^{k_1} \prod_{i=1}^n y_i^{k_2 - 1} \prod_{1 \le i<j \le n} (y_j-y_i)^{k_3}} & & (\text{by Lemma \ref{lem-residue-2}}) \\
&= \CT_{\x} M(n; k_2, k_1, k_3),
\end{align*}
where the last equation holds by making the change of variables $y_i = x_{n+1-i}$ for all $i$.
\end{proof}
The symmetry in Corollary \ref{cor-symMorris} was considered by Morales and Shi \cite{2021-symMorris}. They gave two additional proofs using interpretations of
$\CT_{\x} M(n; k_1, k_2, k_3)$ by the volume of a flow polytope, and by the number of certain integer flows.

\section{The proof of Theorem \ref{theo-main}}\label{s-proof}
Our proof of Theorem \ref{theo-main} relies on two alternative constant term expressions of $h_n(t)$.

The first one is obtained by making the change of variables $x_i=y_1 y_2 \cdots y_i$ for all $i$ in Lemma \ref{lem-changevars}.
We can write \eqref{e-CTH21} as follows:
\begin{equation}\label{e-hn-resi-mul}
\begin{aligned}
 h_n(t) &=  \CT_{\y} \bar \H_n|_{x_i = y_1 y_2 \cdots y_i}  \\
 &= \CT_{\y} \frac{ y_1^{-t+n-2} y_n^{-t}  \prod_{i=2}^{n-1} y_i^{-t + (i-1) (n-i)} \prod_{i=2}^{n-1} (1 - y_{i+1} \cdots y_n)}{(1-y_1 y_2 \cdots y_n)^2 \prod_{i=2}^{n-1} (1-y_1 y_2 \cdots y_i)^{3} \prod_{2 \le i<j \le n} (1 - y_{i+1} \cdots y_j)^{2}}.
 \end{aligned}
\end{equation}
Here we shall explain the series expansion as follows. Since the change of variables transforms
$$\frac{1}{1-x_j/x_i}=\sum_{k\ge 0} (x_j/x_i)^k   \mapsto  \sum_{k\ge 0} (y_{i+1} \cdots y_j)^k = \frac{1}{1-y_{i+1} \cdots y_j}$$
we can simply treat $0 < y_i<1$ in the new constant term expression.
Later we will also use the change of variables $y_1 =x_1$ and $y_i=x_i/x_{i-1}$ for all $i \ge 2$ to transform back. Then the series expansion will be clear.

Equation \eqref{e-hn-resi-mul} has three consequences, as stated in the next three lemmas.
\begin{lem}\label{lem-part3-1}
Let $N = \big\lfloor \frac{(n-1)^2}{4} \big\rfloor$. Then $h_n(t) = 0$ for any $0 \le t \le N-1$.
\end{lem}
\begin{proof}
Consider \eqref{e-hn-resi-mul}. Let
\[
p_i = \begin{cases}
        -t + n - 2, & \mbox{if } i=1; \\
        -t + (i-1) (n-i), & \mbox{if } 2 \le i \le n-1; \\
        -t, & \mbox{if } i=n.
      \end{cases}
\]
Clearly $p_i \le p_r = -t + (r-1)(n-r) = -t + N$ for any $1 \le i \le n$, where $r = \big\lfloor \frac{n+1}{2} \big \rfloor$.
If $0\leq t \le N-1$, then $p_r > 0$, i.e., $\bar \H_n|_{x_i = y_1 y_2 \cdots y_i}$ only contains positive powers in $y_r$, so that $\CT_{y_r} \bar \H_n|_{x_i = y_1 y_2 \cdots y_i} = 0$.
Hence $h_n(t) = 0$.
\end{proof}

\begin{lem}\label{lem-part4-1}
Let $r$ be a positive integer. If $n=2r$, then
\[
h_n(r(r-1)) = \big(\CT_{\x} M(r - 1; 3, 1, 2) \big)^2.
\]
\end{lem}
\begin{proof}
By \eqref{e-hn-resi-mul}, we have
\[
h_n(r(r-1)) = \CT_{\y} \frac{ y_1^{-(r-1)(r-2)} y_n^{-r(r-1)} \prod_{i=2}^{n-1} y_i^{-(r-i)(r-i+1)} \prod_{i=2}^{n-1} (1 - y_{i + 1} \cdots y_n)}{(1-y_1 y_2 \cdots y_n)^{2} \prod_{i=2}^{n-1} (1-y_1 y_2 \cdots y_i)^{3} \prod_{2 \le i<j \le n} (1-y_{i+1} \cdots y_j)^{2}}.
\]
Since $-(r-i)(r-i+1)=0$ when $i=r$ or $i=r+1$, $h_n(r(r-1))$ is a power series in both $y_{r}$ and $y_{r+1}$. We can eliminate $y_r$ and $y_{r+1}$ by substituting $y_r = y_{r+1} = 0$ directly. Then using the change of variables $y_1 =x_1$ and $y_i=x_i/x_{i-1}$ for all $i \ge 2$, $h_n(r(r-1))|_{y_r = y_{r+1} = 0}$ is naturally divided into two parts. More precisely, $h_n(r(r-1))= F_1 \cdot F_2$, where
\begin{align*}
F_1 &= \CT_{\x} \frac{1}{\prod_{i=2}^{r-1} x_i^2 \prod_{i=2}^{r-1} (1 - x_i)^3 \prod_{2 \le i<j \le r-1} (x_i - x_j)^2}, \\
F_2 &= \CT_{\x} \frac{x_n^{-r(r-1)} \prod_{i=r+1}^{n-1} x_i^{n - 1} \prod_{i=r+1}^{n-1} (x_i - x_n)}{\prod_{r + 1 \le i < j \le n} (x_i - x_j)^2}.
\end{align*}
Obviously
\[
F_1 = \CT_{\x} M(r-2; 3, 3, 2) = \CT_{\x} M(r-1, 3, 1, 2),
\]
where the last equation holds by Lemma \ref{lem-Morris}.
For $F_2$, by making the change of variables $x_i= y_i^{-1}$ for all $i$ in Lemma \ref{lem-changevars} and then making $y_i = x_{n-i}$, we obtain
\[
F_2 = \CT_{\y} \frac{1}{\prod_{i=r+1}^{n-1} y_i^2 \prod_{i=r+1}^{n-1} (1-y_i) \prod_{r + 1 \le i < j \le n - 1} (y_i - y_j)^2} = \CT_{\x} M(r-1; 3, 1, 2).
\]
Here we clarify the series expansion under the change of variables $x_i\to  y_i^{-1}\to x_{n-i}$:
$$\frac{1}{1-x_j/x_i}=\sum_{k\ge 0} (x_j/x_i)^k   \mapsto  \sum_{k\ge 0} (y_i/y_j)^k  \mapsto  \sum_{k\ge 0} (x_{n-i}/x_{n-j})^k= \frac{1}{1-x_{n-i}/x_{n-j}}.$$
Thus in the new constant term, we can still use the order  $1 > x_1 > x_2 > \cdots > x_n >0$.
This completes the proof.
\end{proof}

\begin{lem}\label{lem-part4-2}
Let $r$ be a positive integer. If $n=2r-1$, then
\[
h_n((r-1)^2) = \big(\CT_{\x} M(r - 1; 2, 1, 2) \big)^2.
\]
\end{lem}
\begin{proof}
By \eqref{e-hn-resi-mul}, we have
\[
h_n((r-1)^2) = \CT_{\y} \frac{ y_1^{-(r-2)^2} y_n^{-(r-1)^2} \prod_{i=2}^{n-1} y_i^{-(r-i)^2} \prod_{i=2}^{n-1} (1 - y_{i + 1} \cdots y_n)}{(1-y_1 y_2 \cdots y_n)^{2} \prod_{i=2}^{n-1} (1-y_1 y_2 \cdots y_i)^{3} \prod_{2 \le i<j \le n} (1-y_{i+1} \cdots y_j)^{2}}.
\]
Since $-(r-i)^2 = 0$ only when $i=r$, $h_n((r-1)^2)$ is a power series in $y_{r}$. Then the remaining work proceeds similarly to the proof of Lemma \ref{lem-part4-1}.
\end{proof}

\medskip
The second one is obtained by Corollary \ref{cor-barH2xi1} and Lemma \ref{lem-residue-2}. We can also write \eqref{e-CTH21} as follows:
\begin{equation}\label{e-hn-resi-1}
h_n(t) =  \CT_{\x} \frac{(1-x_1)^{-t-1} \prod_{i=2}^{n-1} (1 - x_i)^{n-2} \prod_{i=2}^{n-1} (x_1-x_i)}{x_1 \prod_{i=2}^{n-1} x_i^2 \prod_{1 \le i<j \le n-1} (x_i-x_j)^{2}},
\end{equation}
where we make the change of variables $y_i = x_{n+1-i}$ after applying Lemma \ref{lem-residue-2}. It should be clear that we shall still treat $1 > x_1 > x_2 > \cdots > x_{n-1} > 0$ when expanded as Laurent series.

The formula in \eqref{e-hn-resi-1} also yields several results. Lemma \ref{lem-part1} states that $h_n(t)$ is a polynomial in $t$, which allows us to
extend the definition of $h_n(t)$ for negative $t$. Both Lemmas \ref{lem-part3-2} and \ref{lem-part2} are based on this representation.

\begin{lem}\label{lem-part1}
For a positive integer $n$, $h_n(t)$ is a polynomial in $t$ of degree $(n-1)^2$ with leading coefficient $ \frac{1}{(n-1)^2!}\CT_{\x} M(n-1; 1, 1, 2)$.
\end{lem}
\begin{proof}
We make the change of variables $x_i = y_1 y_2 \cdots y_i$ for all $i$ again. Thus \eqref{e-hn-resi-1} becomes $h_n(t) = \CT_{\y} F_1 \cdot F_2$, where
\begin{align*}
F_1 &= \frac{(1-y_1)^{-t-1} \prod_{i=2}^{n-1} (1 - y_1 y_2 \cdots y_i)^{n-2}}{y_1^{(n-1)^2}}, \\
F_2 &= \frac{\prod_{i=2}^{n-1} (1 - y_2 \cdots y_i)}{\prod_{i=2}^{n-1} y_i^{(n-i)(n-i+1)} \prod_{1 \le i<j \le n-1} (1 - y_{i+1} \cdots y_j)^{2}}.
\end{align*}
Observe that
\[
Q(t) = \CT_{y_1} F_1 = \big[ y_1^{(n-1)^2} \big] \Bigg( \sum_{k=0}^{(n-1)^2} \binom{t+k}{k} y_1^k \Bigg) \prod_{i=2}^{n-1} \Bigg( \sum_{k=0}^{n-2} \binom{n-2}{k} (-y_2 \cdots y_i)^k y_1^k \Bigg)
\]
is a polynomial in $t$ of degree $(n-1)^2$.
The coefficients $[t^\ell]\; Q(t)$ are polynomials in $y_2, \dots, y_{n-1}$. In particular, the leading coefficient $[t^{(n-1)^2}]\; Q(t)$ is $\frac{1}{(n-1)^2!}$. Therefore,
$$h_n(t) = \sum_{\ell=0}^{(n-1)^2}  t^{\ell} \cdot  \CT_{\y} \big([t^\ell]\; Q(t) \big) F_2 $$
is a polynomial in $t$ of degree $(n-1)^2$.

For the leading coefficient, we make the change of variables $y_2 = x_2$ and $y_i = x_i/x_{i-1}$ for $i > 2$ to obtain
\begin{align*}
(n-1)^2! \cdot [t^{(n-1)^2}] \; h_n(t) &=  \CT_{\y} F_2= \CT_{\x} \frac{1}{\prod_{i=2}^{n-1} x_i^2 \prod_{i=2}^{n-1} (1 - x_i) \prod_{2 \le i<j \le n-1} (x_i - x_j)^{2}}  \\
&= \CT_{\x} M(n-2; 3, 1, 2) \\
&= \CT_{\x} M(n-2; 1, 3, 2) \hspace{3.1cm} (\text{by Corollary \ref{cor-symMorris}}) \\
&= \CT_{\x} M(n-1; 1, 1, 2). \hspace{3cm} (\text{by Lemma \ref{lem-Morris}})
\end{align*}
This completes the proof.
\end{proof}

\begin{lem}\label{lem-part3-2}
For any $-n+1 \le t \le -1$, we have $h_n(t) = 0$.
\end{lem}
\begin{proof}
By \eqref{e-hn-resi-1}, we can write
\begin{align*}
h_n(t) &= \CT_{\x} \frac{\prod_{i=2}^{n-1} (1 - x_i)^{n-2}}{\prod_{i=2}^{n-1} x_i^2 \prod_{2 \le i<j \le n-1} (x_i-x_j)^{2}} \CT_{x_1} \frac{(1-x_1)^{-t-1} }{x_1 \prod_{i=2}^{n-1} (x_1-x_i)} \\
&= \CT_{\x} \frac{\prod_{i=2}^{n-1} (1 - x_i)^{n-2}}{\prod_{i=2}^{n-1} x_i^2 \prod_{2 \le i<j \le n-1} (x_i-x_j)^{2}} \CT_{x_1} \frac{(1-x_1)^{-t-1} x_1^{-n+1}}{ \prod_{i=2}^{n-1} (1-x_i/x_1)}.
\end{align*}
When $-n+1 \le t \le -1$, the expansion of $\frac{(1-x_1)^{-t-1} x_1^{-n+1}}{ \prod_{i=2}^{n-1} (1-x_i/x_1)}$ only contains negative powers in $x_1$. Hence the constant term in $x_1$ equals $0$, which implies that $h_n(t) = 0$.
\end{proof}

\begin{lem}\label{lem-part2}
For a positive integer $n$, $h_n(t) = (-1)^{n-1} h_n(-t-n)$.
\end{lem}
\begin{proof}
By applying Lemma \ref{lem-residue-1} to \eqref{e-hn-resi-1} and  then making the change of variables $y_i = x_i$ for all $i$, we obtain
\[
h_n(t) = (-1)^{n-1} \CT_{\x} \frac{(1-x_1)^{t+n-1} \prod_{i=2}^{n-1} (1 - x_i)^{n-2} \prod_{i=2}^{n-1} (x_1-x_i)}{x_1 \prod_{i=2}^{n-1} x_i^2 \prod_{1 \le i<j \le n-1} (x_i-x_j)^{2}}.
\]
Taking $t = -t - n$ in the above formula and comparing with \eqref{e-hn-resi-1} yield $h_n(t) = (-1)^{n-1} h_n(-t - n)$.
\end{proof}

Now we conclude this section by giving a proof of our main result.
\begin{proof}[Proof of Theorem \ref{theo-main}]
Parts (1) and (2) are Lemmas \ref{lem-part1} and \ref{lem-part2}, respectively.
Part (3) holds by Lemmas \ref{lem-part3-1} and \ref{lem-part3-2} with the application of part (2).
Part (4) follows from Lemmas \ref{lem-part4-1} and \ref{lem-part4-2}.
\end{proof}

\section{The computation of $h_n(t)$}\label{s-recursion}
In this section, we first briefly introduce several formulas about elementary symmetric functions.
See \cite{book-Mac-SymFun} or \cite{book-Stanley-EC2} for a comprehensive overview of symmetric functions.
Next we transform $h_n(t)$ into $D_n(0, t + n - 1, n - 2 , 1, 2)$, where $D_n(\ell,t,k_1,k_2,k_3)$ is defined in \eqref{e-def-Dn-1} involving elementary symmetric functions. Using the well-known fact $\res_{x_i} \partial_{x_i} F=0$, we derive a relation among $D_n(\ell,t,k_1,k_2,k_3)$ (See Proposition \ref{prop-equation-ell}). Furthermore, for a fixed $n$, we can obtain a recursion for $D_n(0, t, n - 2 , 1, 2)$ (and hence of $h_n(t)$) with respect to $t$ by the method of undetermined coefficients. Finally, we can construct $h_n(t)$ using the recursion and Theorem \ref{theo-main}.

\subsection{Basic formulas about elementary symmetric functions}\label{ss-SymFunc}
Let $X := x_1+x_2+\cdots +x_n$, and let
\[
e_{\ell}[X] := \sum_{1\le i_1 < i_2 < \cdots < i_{\ell} \le n} x_{i_1} x_{i_2} \cdots x_{i_\ell}
\]
be the $\ell$-th elementary symmetric function on the variables $x_1, x_2, \dots, x_n$. Traditionally, we set $e_{0}[X] = 1$, and $e_{\ell}[X] = 0$ if $\ell < 0$ or $\ell > n$.
We need the well-known formula
\begin{equation}\label{e-hasxi}
e_{\ell}[X] = e_{\ell}[X - x_i] + x_i e_{\ell - 1} [X - x_i].
\end{equation}

We also need the following known formulas, which will be frequently used later.
\begin{lem}\label{lem-eFormula}
For any $0 \le \ell \le n$, we have
\begin{align}
\sum_{i=1}^n e_{\ell}[X-x_i] &=(n-\ell) e_{\ell}[X], \label{e-sumel}\\
\sum_{i=1}^n x_i e_{\ell}[X-x_i] &=(\ell+1)e_{\ell+1}[X], \label{e-sumxel} \\
\sum_{1\leq i\neq j \leq n} e_{\ell}[X-x_i-x_j] &=(n-1-\ell)(n-\ell) e_\ell[X].  \label{e-sumelj}
\end{align}
\end{lem}
\begin{proof}
Observe that $e_\ell[X-x_i]$ contains the term $x_{i_1} x_{i_2} \cdots x_{i_\ell}$ if and only if $i \notin \{i_1, i_2, \dots, i_{\ell}\}$ for any $1 \le i \le n$ and $1\le i_1 < i_2 < \cdots < i_{\ell} \le n$. Thus each term of $e_{\ell}[X]$ appears $n-\ell$ times in $\sum_{i=1}^n e_{\ell}[X-x_i]$. Then equation \eqref{e-sumel} holds.

Equation \eqref{e-sumxel} holds by an analogous reasoning, that is, $x_i e_\ell[X-x_i]$ contains the term $x_{i_1} x_{i_2} \cdots x_{i_{\ell + 1}}$ if and only if $i \in \{i_1, i_2, \dots, i_{\ell + 1}\}$ for any $1 \le i \le n$ and $1\le i_1 < i_2 < \cdots < i_{\ell+1} \le n$.

Equation \eqref{e-sumelj} can be obtained using \eqref{e-sumel} twice. More precisely,
\begin{align*}
\sum_{1\leq i\neq j \leq n} e_{\ell}[X-x_i-x_j] &= \sum_{i=1}^{n} \sum_{j=1, j \neq i}^{n} e_{\ell}[(X-x_i) - x_j] \\
&= \sum_{i=1}^{n} (n-1-\ell) e_\ell[X-x_i] = (n-1-\ell)(n-\ell) e_\ell[X].
\end{align*}
This completes the proof.
\end{proof}

Denote by
\begin{equation}\label{e-Fxixj}
  F(x_i,x_j) := \frac{x_i(1-x_i)}{x_i-x_j} e_\ell[X-x_i].
\end{equation}
We need the following lemma.
\begin{lem}\label{lem-sumFij}
For any distinct $i,j \in \{1, 2, \dots, n\}$, let $F(x_i, x_j)$ be as in \eqref{e-Fxixj}. Then we have
\[
F(x_i,x_j) + F(x_j,x_i) = e_\ell[X-x_i-x_j] + e_{\ell+1}[X-x_i-x_j] - e_{\ell+1}[X].
\]
\end{lem}
\begin{proof}
We use the short notation $\hat{X}=X-x_i-x_j$ in this proof.
By \eqref{e-hasxi}, we can write $e_\ell[X-x_i]=e_\ell[\hat{X}] + x_j e_{\ell-1}[\hat{X}]$. Then we obtain
\begin{align*}
F(x_i,x_j)+F(x_j,x_i) &= e_\ell[\hat X]\left( \frac{x_i(1-x_i)}{x_i-x_j}+\frac{x_j(1-x_j)}{x_j-x_i}\right) \\
&\quad + e_{\ell-1}[\hat X] \left( \frac{x_ix_j(1-x_i)}{x_i-x_j}+\frac{x_j x_i(1- x_j)}{x_j-x_i}\right)\\
&=(1 - x_i - x_j) e_\ell[\hat X] - x_i x_j e_{\ell-1}[\hat X]\\
&=e_\ell[\hat X] + e_{\ell+1}[\hat X] - e_{\ell+1}[X],
\end{align*}
where the last equation holds by $e_{\ell+1}[X]=e_{\ell+1}[\hat X] +e_\ell[\hat X]e_1[x_i+x_j]+e_{\ell-1}[\hat X]e_2[x_i+x_j].$
\end{proof}

\subsection{The transformation for $h_n(t)$}\label{ss-transform}
For $k_1,k_2,k_3\in \NN$, we denote by
\begin{equation}\label{e-def-Phi}
\Phi_n(\ell, t,k_1,k_2,k_3) := \frac{e_\ell[Y-y_n]\prod_{i=1}^{n} x_i^{k_1} \prod_{i=1}^{n-1} (x_i-x_n)}{x_n^t \prod_{i=1}^n (1-x_i)^{k_2} \prod_{1\le i<j \le n} (x_i-x_j)^{k_3}},
\end{equation}
where $y_i = 1-x_i$ are treated as variables, and $Y := y_1 + y_2 + \cdots + y_n$.
Define
\begin{equation}\label{e-def-Dn-1}
D_n(\ell,t,k_1,k_2,k_3) := \res_{\x} \Phi_n(\ell, t,k_1,k_2,k_3).
\end{equation}
By the definition of $\Phi_n(\ell, t, k_1, k_2, k_3)$, we have
\begin{enumerate}
  \item[(1)] $\Phi_n(\ell, t, k_1, k_2, k_3) = 0$ if $\ell < 0$ or $\ell > n-1$;

  \item[(2)] $ x_n^i \Phi_n(\ell, t, k_1, k_2, k_3) = \Phi_n(\ell, t-i, k_1, k_2, k_3)$ for any integer $i$.
\end{enumerate}

The following lemma transforms the computation of $h_n(t)$ into $D_n(0, t + n - 1, n-2 , 1, 2)$.
\begin{lem}\label{lem-transform}
For a positive integer $n$, let $h_n(t)$ and $D_n(\ell, t, k_1, k_2, k_3)$ be as in \eqref{e-CTH21} and \eqref{e-def-Dn-1}, respectively. Then we have
\[
h_n(t) = D_n(0, t + n - 1, n - 2 , 1, 2).
\]
\end{lem}
\begin{proof}
By \eqref{e-def-Dn-1}, we have
\[
D_n(0, t + n - 1, n-2 , 1, 2) = \res_{\x} \Phi_{n}(0, t + n - 1, n-2 , 1, 2) = \CT_{\x} f,
\]
where
\[
f= \Phi_{n}(0, t + n - 1, n-2 , 1, 2) \prod_{i=1}^{n} x_i = \frac{\prod_{i=1}^{n-1} x_i^{n-1} \prod_{i=1}^{n-1} (x_i-x_n)}{x_n^{t} \prod_{i=1}^n (1-x_i) \prod_{1\le i<j \le n} (x_i-x_j)^{2}}.
\]
Since $f$ is proper in $x_1$, by Lemmas \ref{lem-contri} and \ref{lem-Ai0}, we have
\[
\CT_{x_1} f = \CT_{x_1 = 1} f = ((1-x_1) f) \big|_{x_1 = 1}.
\]
Then the lemma follows by checking that $((1-x_1) f) \big|_{x_1 = 1}$ is completely the same as \eqref{e-barH21}.
\end{proof}

\subsection{A recursion for $h_n(t)$}
We first derive a relation among $D_n(\ell,t,k_1,k_2,k_3)$ for fixed $k_1,k_2,k_3$.
Then taking $k_1 = n-2$, $k_2=1$ and $k_3=2$ to give a specific relation among $D_n(\ell, t, n-2 , 1, 2)$.
Furthermore, we derive a recursion for $D_n(0, t, n-2 , 1, 2)$ with respect to $t$ by the method of undetermined coefficients.

Let $\Phi_n(\ell, t, k_1, k_2, k_3)$ be as in \eqref{e-def-Phi}. Denote by $V := \Phi_n(0, t, k_1, k_2, k_3)$.
Then
\begin{equation}\label{e-def-Dn-2}
D_n(\ell,t,k_1,k_2,k_3) = \res_{\x} e_\ell[Y-y_n] V.
\end{equation}

\begin{prop}\label{prop-equation-ell}
Let $k_1,k_2,k_3\in \NN$, and $D_n(\ell,t,k_1,k_2,k_3)$ be as in \eqref{e-def-Dn-2}. Then we have
\begin{align*}
A_{\ell+1}  D_n(\ell+1, t, k_1, k_2, k_3)+ A_{\ell}D_n(\ell, t, k_1, k_2, k_3) &+  \bar A_{\ell} D_n(\ell, t - 1, k_1, k_2, k_3) \\
+ A_{\ell-1}D_n(\ell-1, t, k_1, k_2, k_3) &+ \bar A_{\ell-1} D_n(\ell-1, t-1, k_1, k_2, k_3) = 0,
\end{align*}
where
\begin{align*}
A_{\ell+1} &= (\ell+1)\Big(k_2 - k_1 - 3 +\frac{k_3}2 (2n-2-\ell)\Big), \\
A_{\ell} & = t - n + (\ell+1) \Big(k_2-k_1-2 + \frac{k_3}{2}(2n -2 - \ell)\Big) - (n-\ell)\Big(k_2 - 2 + \frac{k_3}{2}(n-1-\ell)\Big), \\
\bar A_{\ell} & = -\Big(t - n + 1 + (\ell+1)\Big(k_2-k_1-2 + \frac{k_3}{2}(2n -2 - \ell)\Big)\Big),  \\
A_{\ell-1} &= - \bar A_{\ell-1} = - (n-\ell) \Big(k_2 - 1 + \frac{k_3}{2} (n -1-\ell) \Big).
\end{align*}
\end{prop}
\begin{proof}
Keep in mind that $y_i = 1 - x_i$.
Since $ e_\ell[Y-y_i] x_i(1-x_i) V$ is a formal Laurent series in $x_i$ for any $1 \le i \le n$, we have $\res_{x_i} \partial_{x_i} e_\ell[Y-y_i] x_i(1-x_i) V = 0$. Hence
\[
\res_{\x} \sum_{i=1}^{n} \partial_{x_i} e_\ell[Y-y_i] x_i(1-x_i) V = 0.
\]
It can be computed that
\begin{align*}
& \partial_{x_i} e_\ell[Y-y_i] x_i (1 - x_i) V \\
  &= \begin{cases}
       e_\ell[Y-y_i] V \Big((k_1+1)(1 - x_i) + (k_2-1) x_i - \frac{x_i(1-x_i)}{x_n - x_i} + k_3\sum\limits_{j=1,j\neq i}^{n} \frac{x_i(1-x_i)}{x_j-x_i}   \Big), & \mbox{if } i < n; \\
       e_\ell[Y-y_n] V \Big((k_1+1-t) (1 - x_n) + (k_2-1) x_n + (k_3-1) \sum\limits_{j=1}^{n-1} \frac{x_n(1-x_n)}{x_j - x_n}  \Big), & \mbox{if } i=n.
     \end{cases}
\end{align*}
Now we take the sum over all $i$. The computation is naturally divided into three parts, denoted
$S_1$, $S_2$, $S_3$ respectively. In detail,
\begin{align*}
S_1 &= -t (1 - x_n) e_\ell[Y-y_n] + \sum_{i=1}^n (k_1+1)(1 - x_i) e_\ell[Y-y_i] \\
&= -t y_n e_\ell[Y-y_n] + \sum_{i=1}^n (k_1+1)y_i e_\ell[Y-y_i] & &(\text{by $y_i = 1-x_i$}) \\
&= - t y_n e_\ell[Y-y_n] + (k_1+1)(\ell+1)e_{\ell+1}[Y], & &(\text{by \eqref{e-sumxel}})\\
S_2 &= \sum_{i=1}^n (k_2-1) x_i e_\ell[Y-y_i] = \sum_{i=1}^n (k_2-1)(1-y_i)e_\ell[Y-y_i] & &(\text{by $y_i = 1-x_i$})\\
&= (k_2-1)\big((n-\ell) e_\ell[Y] - (\ell+1)e_{\ell+1}[Y]\big), & &(\text{by \eqref{e-sumel} and \eqref{e-sumxel}})
\end{align*}
and
\begin{align*}
S_3 &= k_3 \sum_{i=1}^n \sum_{j=1,j\neq i}^n \frac{x_i(1-x_i)}{x_j-x_i} e_{\ell}[Y-y_i] - \sum_{i=1}^{n-1} \frac{x_i(1-x_i)}{x_n - x_i} e_{\ell}[Y-y_i] - \sum_{j=1}^{n-1} \frac{x_n(1-x_n)}{x_j - x_n} e_{\ell}[Y-y_n] \\
&= k_3 \sum_{i=1}^n \sum_{j=1,j\neq i}^n F(y_i, y_j) - \sum_{i=1}^{n-1} F(y_i, y_n) - \sum_{j=1}^{n-1} F(y_n, y_j) \hspace{1.9cm} (\text{by $y_i = 1-x_i$})\\
&= \frac{k_3}{2} \sum_{1\leq i\neq j \leq n} (F(y_i,y_j) + F(y_j,y_i)) - \sum_{i=1}^{n-1} (F(y_i,y_n) + F(y_n,y_i)) \\
&= \frac{k_3}2 \sum_{1\leq i\neq j \leq n} (e_\ell[Y-y_i-y_j] - e_{\ell+1}[Y] + e_{\ell+1}[Y-y_i-y_j]) \\
& \quad -\sum_{i=1}^{n-1} (e_\ell[Y-y_i-y_n] - e_{\ell+1}[Y] + e_{\ell+1}[Y-y_i-y_n])\hspace{1.9cm} (\text{by Lemma \ref{lem-sumFij}}) \\
&= \frac{k_3}2 \big((n-\ell)(n-1-\ell) e_\ell[Y] - (\ell+1)(2n-2-\ell) e_{\ell+1}[Y]\big) - (n-1-\ell) e_{\ell}[Y-y_n]\\
& \quad + (n-1)e_{\ell+1}[Y] - (n-2-\ell)e_{\ell+1}[Y-y_n]. \hspace{3.4cm} (\text{by \eqref{e-sumel} and \eqref{e-sumelj}})
\end{align*}
By \eqref{e-hasxi}, we have $e_{\ell+1}[Y] = e_{\ell+1}[Y-y_n] + y_n e_{\ell}[Y-y_n]$, which yields
\[
- (S_1 + S_2 + S_3) = A_{\ell+1} e_{\ell+1}[Y - y_n] + (A_{\ell} + \bar A_{\ell} x_n) e_{\ell}[Y - y_n] +(A_{\ell-1} + \bar A_{\ell-1} x_n) e_{\ell-1}[Y - y_n].
\]
Then the lemma follows by $\res_{\x} (S_1 + S_2 + S_3) V  = 0$ and $x_n \Phi_n(\ell, t, k_1, k_2, k_3) = \Phi_n(\ell, t-1, k_1, k_2, k_3)$.
\end{proof}

For convenience, we denote by $d_n(\ell, t) := D_n(\ell, t, n - 2 , 1, 2)$.
Substituting $k_1 = n-2$, $k_2 = 1$, $k_3 = 2$ into Proposition \ref{prop-equation-ell} yields
\begin{equation}\label{e-Dn-recur}
\begin{aligned}
&(\ell+1) (n- 2 -\ell)d_n(\ell+1, t) + (t - (n - 2 \ell) ( n - 2 - \ell) - 1) d_n(\ell, t) \\
&- (t + \ell ( n - 2 - \ell))d_n(\ell, t-1) - (n - \ell) (n - 1 - \ell) d_n(\ell-1, t)  \\
&+ (n - \ell) (n - 1 - \ell)d_n(\ell-1, t - 1) = 0.
\end{aligned}
\end{equation}

The remaining work is to obtain a recursion for $d_n(0, t)$ with respect to $t$, which can be completed by the following steps:
\begin{enumerate}
  \item We obtain $n-1$ equations, say $E_0, E_1, \dots, E_{n-2}$, by substituting $\ell = 0, 1, \dots, n-2$ into \eqref{e-Dn-recur} respectively. For any $\ell \le -1$ and $\ell \geq n-1$, \eqref{e-Dn-recur} becomes $0 = 0$. The only nontrivial case is when $\ell = n-1$,
      for which we can verify $d_n(n-1, t)=0$ as follows:
      \begin{align*}
      d_n(n-1, t) &= \res_{\x} \Phi_n(n-1, t, n-2, 1, 2) \\
      &= \CT_{\x} \frac{\prod_{i=1}^{n} x_i^{n-1} \cdot \prod_{i=1}^{n-1} (x_i - x_n)}{x_n^{t} (1 - x_n) \prod_{1\le i < j \le n} (x_i-x_j)^2} \\
      &= \CT_{\x} \frac{\prod_{i=2}^{n} x_i^{n-1} \cdot \prod_{i=2}^{n-1} (x_i - x_n)}{x_n^{t} (1 - x_n) \prod_{2\le i < j \le n} (x_i-x_j)^2} \CT_{x_1} \frac{x_1^{-n+2}}{(1-x_n/x_1)\prod_{i=2}^{n-1} (1-x_i/x_1)^2}.
      \end{align*}
      Since $n \ge 3$, the rightmost part only contains negative powers in $x_1$, which yields $d_n(n-1, t) = 0$.

  \item For each $\ell \in \{0, 1, \dots, n-2\}$, we substitute $t=t-i$ for $i = 0, 1, \dots, n - 2 - \ell$ into $E_{\ell}$ respectively to obtain $\binom{n}{2}$ equations:
      \[
      E_{0,0}, E_{0,1}, \dots, E_{0, n-2}, E_{1, 0}, E_{1, 1}, \dots, E_{1,n-3}, \dots, E_{n-2, 0},
      \]
      where $E_{\ell, i} = E_{\ell} |_{t=t-i}$.

  \item We obtain a recursion for $d_n(0, t)$ with respect to $t$ using the method of undetermined coefficients. That is, we consider the equation
      \begin{equation}\label{e-undeter}
      \sum_{\ell = 0}^{n-2} \sum_{i=0}^{n-2-\ell} a_{\ell, i} E_{\ell, i}
      \end{equation}
      with undetermined coefficients the $a_{\ell, i}$'s such that $d_n(\ell, t-i)$ disappear for all $\ell > 0$.
      This gives a linear system described as follows:
      \begin{equation}\label{e-hasequation}
      B (a_{0,0}, a_{0,1}, \dots, a_{0, n-2}, a_{1, 0}, a_{1, 1}, \dots, a_{1,n-3}, \dots, a_{n-3, 0}, a_{n-3, 1})^T = a_{n-2,0} \b,
      \end{equation}
      where $B$ is a $\frac{(n+1)(n-2)}{2} \times \frac{(n+1)(n-2)}{2}$ upper triangular matrix with $B_{i,i} = (\ell+1) (n- 2 -\ell)$ if $\sum_{j=0}^{\ell-1} (n-1-j) < i \le \sum_{j=0}^{\ell} (n-1-j)$.
      That is, all the diagonal entries of $B$ are not zero. Hence the linear system in \eqref{e-hasequation} must have a solution. Substituting any nonzero particular solution (say, with respect to $a_{n-2,0}=1$) into \eqref{e-undeter} yields an equation in $d_n(0, t), d_n(0, t-1), \dots, d_n(0,t-n+1)$. 
\end{enumerate}

We compute the recursion for $d_4(0, t)$ in the following example to explain the above steps.
\begin{ex}\label{ex-neq4}
Fix $n=4$. By \eqref{e-Dn-recur}, we have
\begin{align*}
(\ell+1) (2 -\ell)d_4(\ell+1, t) &+ (t - 2(2 - \ell)^2 - 1) d_4(\ell, t) - (t + \ell ( 2 - \ell)) d_4(\ell, t-1) \\
& - (4 - \ell) (3 - \ell)d_4(\ell-1, t) + (4 - \ell) (3 - \ell) d_4(\ell-1, t - 1) = 0.
\end{align*}
Substituting $\ell = 0, 1, 2$ into the above equation respectively gives
\begin{align*}
2 d_4(1, t) + (t - 9) d_4(0, t) - t d_4(0, t-1) &= 0, & &(E_0)  \\
2 d_4(2, t) + (t - 3) d_4(1, t) - (t + 1) d_4(1, t-1) - 6 d_4(0, t) + 6 d_4(0, t - 1) &= 0, & &(E_1)\\
(t - 1) d_4(2, t) - t d_4(2, t - 1) - 2 d_4(1, t) + 2 d_4(1, t-1) &= 0. & &(E_2)
\end{align*}
For each $\ell \in \{0, 1, 2\}$, we list $E_{\ell, i} = E_{\ell} |_{t=t-i}$ for $i = 0, 1, \dots, 2 - \ell$ as follows:
\begin{align*}
2 d_4(1, t) + (t - 9) d_4(0, t) - t d_4(0, t-1) &= 0, & &(E_{0,0}) \\
2 d_4(1, t-1) + (t - 10) d_4(0, t-1) - (t-1) d_4(0, t-2) &= 0, & &(E_{0,1}) \\
2 d_4(1, t-2) + (t - 11) d_4(0, t-2) - (t-2) d_4(0, t-3) &= 0,  & &(E_{0,2})\\
2 d_4(2, t) + (t - 3) d_4(1, t) - (t + 1) d_4(1, t-1) - 6 d_4(0, t) + 6 d_4(0, t - 1) &= 0,  & &(E_{1,0}) \\
2 d_4(2, t-1) + (t - 4) d_4(1, t-1) - t d_4(1, t-2) - 6 d_4(0, t-1) + 6 d_4(0, t - 2) &= 0,  & &(E_{1,1}) \\
(t - 1) d_4(2, t) - t d_4(2, t - 1) - 2 d_4(1, t) + 2 d_4(1, t-1) &= 0.  & &(E_{2,0})
\end{align*}
The equation $\sum_{\ell = 0}^{2} \sum_{i=0}^{2-\ell} a_{\ell, i} E_{\ell, i}$ can be written as
\begin{equation}\label{e-detercoe}
\begin{aligned}
&(2 a_{0,0} + (t-3) a_{1,0} - 2 a_{2,0}) d_4(1, t)  \\
&+ (2 a_{0,1} - (t + 1) a_{1,0} + (t - 4) a_{1,1} + 2 a_{2,0} ) d_4(1, t-1) + (2 a_{0,2} - t a_{1,1} ) d_4(1, t-2) \\
& + (2 a_{1,0} + (t-1) a_{2,0}) d_4(2, t) + (2 a_{1,1} - t a_{2,0}) d_4(2, t-1) \\
&+((t - 9) a_{0,0} - 6 a_{1,0} ) d_4(0, t) + (- t a_{0,0} + (t - 10) a_{0,1} + 6 a_{1,0} - 6 a_{1,1})d_4(0, t-1) \\
&+(- (t-1) a_{0,1} + (t - 11) a_{0,2} + 6 a_{1,1}) d_4(0, t-2) - (t-2) a_{0,2} d_4(0, t-3) = 0.
\end{aligned}
\end{equation}
By letting the coefficients of $d_4(\ell, t-i)$ equal to $0$ for $\ell = 1,2$ and the corresponding $i = 0, 1, \dots, 2 - \ell$, we obtain the following linear system:
\[
\left(\begin{array}{ccccc}
  2 & 0 & 0 & t-3 & 0  \\
  0 & 2 & 0 & -t-1 & t-4  \\
  0 & 0 & 2 & 0 & -t  \\
  0 & 0 & 0 & 2 & 0  \\
  0 & 0 & 0 & 0 & 2
\end{array}\right)
\left(\begin{array}{c}
        a_{0,0} \\
        a_{0,1} \\
        a_{0,2} \\
        a_{1,0} \\
        a_{1,1}
      \end{array}\right)
=a_{2,0} \left(\begin{array}{c}
        2  \\
        -2 \\
        0 \\
        1-t \\
        t
      \end{array}\right)
\]
The solution of this linear system is
\begin{align*}
a_{0,0} &= \frac{t^2 - 4t + 7}{4} a_{2,0}, &
a_{0,1} &= \frac{-2t^2 + 4 t - 3}{4} a_{2,0}, &
a_{0,2} &= \frac{t^2}{4} a_{2,0}, \\
a_{1,0} &= \frac{-t+1}{2} a_{2,0}, &
a_{1,1} &= \frac{t}{2} a_{2,0}.
\end{align*}
Taking $a_{2,0} = 4$ and substituting the above solution into \eqref{e-detercoe} yield
\begin{align*}
(t-3) (t-5)^2 d_4(0, t) &+ (-3 t^3 + 28 t^2 - 74 t + 42) d_4(0, t-1) \\
&+ (3 t^3 - 17 t^2 + 19 t - 3) d_4(0, t-2) - t^2 (t-2) d_4(0, t-3) = 0.
\end{align*}
Furthermore, by Lemma \ref{lem-transform}, we obtain
\begin{equation}\label{e-recur-4}
\begin{aligned}
t (t-2)^2 h_4(t) &+ (-3 t^3 + t^2 + 13 t - 9) h_4(t-1) \\
&+ (3 t^3 + 10 t^2 - 2 t - 18) h_4(t-2) - (t+1) (t+3)^2 h_4(t-3) = 0.
\end{aligned}
\end{equation}
It is easy to check that
\[
h_4(t) = \frac{(t-1) t^2 (t+1) (t+2) (t+3) (t+4)^2 (t+5)}{60480},
\]
which fits \eqref{e-recur-4}.
\end{ex}

\begin{rem}
A formula for $D_n(\ell, t, k_1, k_2, k_3)$ for $\ell>0$ can be obtained, but is not needed here.
\end{rem}

For general $n$, the recursion for $d_n(0, t)$ can be described as $\sum_{i=0}^{n-1} c_i(t) d_n(0, t-i) = 0$, where $c_i(t)$ is a polynomial in $t$ for each $i$. By Lemma \ref{lem-transform}, we have
\begin{equation}\label{e-hrecur}
 h_n(t) = - \frac{1}{c_0(t+n-1)}\sum_{i=1}^{n-1} c_i(t+n-1) h_n(t-i),
\end{equation}
which gives a recursion for $t> t_0$, where $t_0$ is the maximum integer root (if exists) of $c_0(t+n-1)$.

\begin{ex}\label{ex-h5}
Analogous to Example \ref{ex-neq4}, for $n=5$, we can obtain
\begin{equation}\label{e-h5-recur}
\begin{aligned}
& t (t - 4) (t - 3)^2 h_5(t) - (4 t^4 - 16 t^3 - 22 t^2 + 126 t - 80) h_5(t-1) \\
&+ (6 t^4 + 12 t^3 - 74 t^2 - 80 t + 208) h_5(t-2) \\
&- (4 t^4 + 32 t^3 + 50 t^2 - 106 t - 208) h_5(t-3)+ (t + 4)^2 (t + 5) (t + 1)  h_5(t-4) = 0.
\end{aligned}
\end{equation}
Then we compute the explicit formula for $h_5(t)$ by the method of undetermined coefficients.
By parts (1) and (3) of Theorem \ref{theo-main}, we set
\[
h_5(t) = \Big( \frac{300}{16!} t^4 + a_3 t^3 + a_2 t^2 + a_1 t + a_0 \Big) \prod_{i=-8}^{3} (t-i).
\]
Part (4) of Theorem \ref{theo-main} gives $h_5(4) = 9$.
Substituting $t=5$ into \eqref{e-h5-recur} yields $ 20 h_5(5) - 500 h_5(4) = 0$ since $h_5(1) = h_5(2) = h_5(3) = 0$. Thus $h_5(5) = 225$.
By part (2) of Theorem \ref{theo-main}, we have $h_5(-9) = h_5(4) = 9$ and $h_5(-10) = h_5(5) = 225$.
Solving
\[
\left\{
\begin{aligned}
  h_5(4) &= 479001600 (a_0 + 4 a_1 + 16 a_2 + 64 a_3) + \frac{160}{91} = 9, \\
  h_5(-9) &= 479001600 (a_0 - 9 a_1 + 81 a_2 - 729 a_3) + \frac{32805}{728} = 9, \\
  h_5(5) &= 6227020800 (a_0 + 5 a_1 + 25 a_2 + 125 a_3) + \frac{3125}{56} = 225, \\
  h_5(-10) &= 6227020800 (a_0 - 10 a_1 + 100 a_2 - 1000 a_3) + \frac{6250}{7} = 225
\end{aligned}\right.
\]
gives
\[
a_0 = 0, \quad a_1 = \frac{1}{34871316480}, \quad a_2 = \frac{127}{348713164800}, \quad a_3 = \frac{1}{6974263296}.
\]
Therefore,
\[
h_5(t) = \frac{t (t + 5) (5 t^2 + 25 t + 2)}{348713164800} \prod_{i=-8}^{3} (t-i).
\]

\end{ex}

Computer experiments suggest us to conjecture that $t_0 = \big\lfloor \frac{(n-1)^2}{4} \big \rfloor$. For instance, $t_0=4$ when $n=5$.
To verify this conjecture, we need a formula for $c_0(t)$. This can be obtained by considering a subsystem of \eqref{e-undeter}. Observe that $d_n(\ell,t)$ only appears in $E_{\ell-1,0}$, $E_{\ell,0}$ and $E_{\ell+1,0}$. Thus we only need to consider $\sum_{\ell = 0}^{n-2}  a_{\ell, 0} E_{\ell, 0}$. The coefficient of $d_n(\ell,t)$ in $\sum_{\ell = 0}^{n-2}  a_{\ell, 0} E_{\ell, 0}$ is
\[
\ell (n - 1 - \ell) a_{\ell-1,0} + (t - (n - 2 \ell) (n - 2 - \ell) - 1) a_{\ell,0} - (n - 1 - \ell) (n - 2 - \ell) a_{\ell+1,0}
\]
for each $0 \le \ell \le n-2$.
In particular, the coefficient of $d_n(0,t)$ is
\[
c_0(t) = (t - (n-1)^2) a_{0,0} - (n-1)(n-2) a_{1,0}.
\]

We obtain
\[
B \left(\begin{array}{c}
    a_{0,0} \\
    a_{1,0} \\
    \vdots \\
    a_{n-3,0}
  \end{array}\right) = - a_{n-2,0} \b
\]
by setting the coefficients of $d_n(\ell, t)$ equal to $0$ for any $\ell > 0$, where $\b = (0, \dots, 0, -2,  t-1)^T$, and $B$ is a $(n-2) \times (n-2)$ upper triangular matrix with
\[
B_{i,j} = \begin{cases}
            i (n-1-i), & \mbox{if } i=j; \\
            t - (n-2i) (n-2-i) - 1, & \mbox{if } i=j-1; \\
            - (n-1 - i) (n-2-i), & \mbox{if } i=j-2; \\
            0, & \mbox{otherwise}.
          \end{cases}
\]
Clearly $\det B$ is a constant.
Applying Cramer's Rule and the rules of column transformation, we have
\[
a_{0,0} = \frac{(-1)^{n} a_{n-2,0} \det (\beta_2, B', \b)}{\det B},  \quad
a_{1,0} = \frac{(-1)^{n-1} a_{n-2,0} \det (\beta_1, B', \b)}{\det B},
\]
where $\beta_i$ denotes the $i$-th column of $B$ for $i=1, 2$, and $B'$ is obtained from $B$ by removing the first two columns. Consequently,
\begin{align*}
c_0(t) &= \frac{(-1)^n a_{n-2, 0}}{\det B} \big( (t - (n-1)^2) \det (\beta_2, B', \b) + (n-1)(n-2) \det (\beta_1, B', \b) \big) \notag \\
&= \frac{(-1)^n a_{n-2, 0}}{\det B} \det C,
\end{align*}
where
\begin{equation}\label{e-det}
C = \left(\begin{array}{cccc}
                t - (n-1)^2 & -(n-1)(n-2) & \mathbf{0} & 0 \\
                \beta_1 & \beta_2 & B' & \b
              \end{array}\right).
\end{equation}
This is a $(n-1) \times (n-1)$ matrix whose entries are given by:
\[
C_{i,j} = \begin{cases}
            t - (n-2i + 2) (n-1-i) - 1, & \mbox{if } i = j; \\
            (i-1) (n-i), & \mbox{if } i = j + 1; \\
            - (n-1-i) (n-i), & \mbox{if } i = j - 1; \\
            0, & \mbox{otherwise}.
          \end{cases}
\]
Since $\det C$ is just the characteristic polynomial of the $(n-1) \times (n-1)$ matrix $-C|_{t=0}$, it is a polynomial in $t$ of degree $n-1$.
It is not hard to compute the determinant for given $n$. This leads to the following conjecture.
\begin{conj}\label{conj-detC}
Let $C$ be the matrix in \eqref{e-det}. Then
\[
\det C = \begin{cases}
           (t - n + 1) \prod_{i=0}^{r-2} (t - n + 1 - r(r-1) + i (i+1))^2, & \mbox{if } n = 2 r; \\
           (t - n + 1) (t - n + 1 - (r-1)^2) \prod_{i=1}^{r-2} (t - n + 1 - (r - 1)^2 + i^2)^2, & \mbox{if } n = 2 r - 1.
         \end{cases}
\]
\end{conj}
For example,
\[
\det C = \begin{cases}
           (t - 13) (t - 55)^2 (t - 53)^2 (t - 49)^2 (t - 43)^2 (t - 35)^2 (t - 25)^2, & \mbox{if } n = 14, \\
          (t - 14) (t - 63) (t - 62)^2 (t - 59)^2 (t - 54)^2 (t - 47)^2 (t - 38)^2 (t - 27)^2, & \mbox{if } n = 15.
         \end{cases}
\]

If Conjecture \ref{conj-detC} holds, then
\[
t_0 = \Big\lfloor \frac{(n-1)^2}{4} \Big\rfloor = \begin{cases}
                                                    r(r-1), & \mbox{if } n = 2 r; \\
                                                    (r-1)^2, & \mbox{if } n = 2 r - 1.
                                                  \end{cases}
\]
It follows that parts (3) and (4) of Theorem \ref{theo-main} give initial values of $h_n(t)$ for the recursion \eqref{e-hrecur}.

\section{A conjecture about the generating function of $h_n(t)$}\label{s-conj}

Theorem \ref{theo-Stanley} suggests us to consider properties of the generating function of $h_n(t)$. It is necessary to clarify some definitions.
A real number sequence $(a_0, a_1, \dots, a_m)$ is said to be \emph{palindromic} if $a_i = a_{m-i}$ for all $i$.
It is said to be \emph{log-concave} (short for \emph{logarithmically concave}) if $a_i^2 \ge a_{i-1} a_{i+1}$ for all $1 \le i \le m-1$.
Let $b_i = \frac{1}{\binom{m}{i}} a_i$ for all $i$. Then the sequence $(a_0, a_1, \dots, a_m)$ is said to be \emph{strongly log-concave} if $b_i^2 \ge b_{i-1} b_{i+1}$ for all $1 \le i \le m-1$.
See \cite{book-Stanley-AlgComb} for further details.

We raise the following conjecture.

\begin{conj}\label{conj-h*}
Let $N=\big\lfloor \frac{(n-1)^2}{4} \big\rfloor$ and $\hat N = \big\lfloor \frac{(n-2)^2}{2} \big\rfloor$. Then
\[
h^*_n (y) :=  (1 - y)^{(n-1)^2+1} \sum_{t \ge 0} h_n(t) y^t = y^{N} \sum_{i=0}^{\hat N} a_i y^i,
\]
where the $a_i$'s are positive integers.
The polynomial $h^*_n (y)$ is real-rooted, and $(a_0, a_1, \dots, a_{\hat N})$ is palindromic, unimodal and strongly log-concave.
\end{conj}

We have verified Conjecture \ref{conj-h*} for $3 \le n \le 29$.
For example,
\[
h^*_5 (y) =  3 y^4 (3 y^4 + 24 y^3 + 46 y^2 + 24 y + 3).
\]
The roots of $h^*_5 (y)$ are $0$ (with multiplicity $4$), $- 3 \pm \sqrt{6}$, $- 1 \pm \frac{\sqrt{6}}{3}$.
And the other properties are easy to check.

The data of $h^*_n(y)$ for $3 \le n \le 29$ is also available at \cite{datalink}.

\medskip
\noindent \textbf{Acknowledgments:}
This work was supported by the National Natural Science Foundation of China (No. 12071311).


\begin{thebibliography}{10}
\bibitem{Per}
P. Alexandersson, S. Hopkins and G. Zaimi, \emph{Restricted Birkhoff polytopes and Ehrhart period collapse}, arXiv:2206.02276.

\bibitem{1966-Anand-Hn}
H. Anand, V. C. Dumir and H. Gupta, \emph{A combinatorial distribution problem}, Duke Math. J. 33 (1966), 757--769.

\bibitem{2005-Athanasiadis}
C. A. Athanasiadis, \emph{Ehrhart polynomials, simplicial polytopes, magic squares and a conjecture of Stanley}, J. Reine Angew. Math. 583 (2005), 163--174.

\bibitem{2001-Baldoni-Silva}
W. Baldoni-Silva and M. Vergne, \emph{Residues formulae for volumes and Ehrhart polynomials of convex polytopes}, arXiv:math/0103097v1.

\bibitem{2003-BeckPixton}
M. Beck  and  D. Pixton, \emph{The Ehrhart polynomial of the Birkhoff polytope}, Discrete Comput. Geom. 30 (2003), 623--637.

\bibitem{Clara}
C. S. Chan and D. P. Robbins, \emph{On the volume of the polytope of doubly stochastic matrices}, Experiment. Math. 8 (1999), 291--300.

\bibitem{Loera}
J. A. De Loera, F. Liu and R. Yoshida, \emph{A generating function for all semi-magic squares and the volume of the Birkhoff polytope}, J. Algebr. Comb. 30 (2009), 113--139 .

\bibitem{Diaconis}
P. Diaconis and  A. Gangolli, \emph{Rectangular arrays with fixed margins}, IMA Series on Volumes
in Mathematics and its Applications, \# 72 Springer-Verlag (1995), 15--41.


\bibitem{1973-Ehrhart}
E. Ehrhart, \emph{Sur les carr$\acute{e}$s magiques}, C. R. Acad. Sci. Paris 227 A (1973), 575--577.

\bibitem{2008-Forrester-Selberg}
P. J. Forrester and  S. O. Warnaar, \emph{The importance of the Selberg integral}, Bull. Aust. Math. Soc, 45 (2008), 489--534.

\bibitem{1830-Jacobi}
C. G. J. Jacobi, \emph{De resolutione aequationum per series infinitas}, J. Reine Angew. Math. 6 (1830), 257--286.

\bibitem{book-Mac-SymFun}
I. G. Macdonald, \emph{Symmetric Functions and Hall Polynomials}, Oxford classic texts in the physical sciences, Clarendon Press, 1998.

\bibitem{1960-MacMahon}
P. A. MacMahon, \emph{Combinatory analysis}, Chelsea Publishing Co., New York, 1960. MR 25 \# 5003.

\bibitem{1971-gtheorem}
P. McMullen, \emph{The number of faces of simplicial polytopes}, Israel J. Math. 9 (1971), 559--570.

\bibitem{2021-symMorris}
A. H. Morales and W. Shi, \emph{Refinements and Symmetries of the Morris identity for volumes of flow polytopes}, Comptes Rendus. Mathématique, 359 (2021), 823--851.

\bibitem{1982-Morris}
W. G. Morris, \emph{Constant term identities for finite and affine root system}, Ph.D. Thesis, University of Wisconsin, Madison, 1982.

\bibitem{John}
J. Mount, \emph{Fast unimodular counting}, Combin. Probab. Comput. 9 (2000), 277--285.

\bibitem{1944-Selberg}
A. Selberg, \emph{Bemerkninger om et multipelt integral}, Nordisk Mat. Tidskr. 26 (1944), 71--78.

\bibitem{1973-Stanley-Hn}
R. P. Stanley, \emph{Linear homogeneous Diophantine equations and magic labelings of graphs}, Duke Math. J. 40 (1973), 607--632.

\bibitem{1976-Stanley-Hn}
R. P. Stanley, \emph{Magic labelings of graphs, symmetric magic squares, systems of parameters, and Cohen-Macaulay rings}, Duke Math. J. 43 (1976), 511--531.

\bibitem{1983-Stanley}
R. P. Stanley, Combinatorics and Commutative Algebra, Progress in Mathematics 41, Birkh$\ddot{\text{a}}$user, Boston, first edition, 1983; second edition, 1996.

\bibitem{book-Stanley-EC2}
R. P. Stanley, \emph{Enumerative Combinatorics (Volume 2)}, Cambridge studies in advanced mathematics, Cambridge, 2012.

\bibitem{book-Stanley-AlgComb}
R. P. Stanley, \emph{Algebraic Combinatorics: Walks, Trees, Tableaux, and More}, Springer, New York, 2013.

\bibitem{2004-Xin-phd}
G. Xin, \emph{The Ring of Malcev-Neumann Series and The Residue Theorem}, Ph.D. thesis, Brandeis University, arXiv:math.CO/0405133, 2004.

\bibitem{datalink}
G. Xin and C. Zhang, \emph{Morris-Data}, \texttt{https://pan.baidu.com/s/16zjRkvXMQfhDK7E7fBhZtQ} with passcode \texttt{morr}.

\bibitem{2022-CTTypeA}
G. Xin, C. Zhang, Y. Zhou and Y. Zhong, \emph{The constant term algebra of type $A$: the structure}, arXiv: 2209.15495, 2022.

\bibitem{1999-Zeilberger}
D. Zeilberger, \emph{Proof of a conjecture of Chan, Robbins, and Yuen}, Electron. Trans. Numer. Anal, 9 (1999), 147--148.
\end{thebibliography}
\end{document}